\documentclass[a4paper,12pt]{article}
\usepackage{preamble}

\bibliography{references}

\title{Compact-Open Dualities\\for Stably Continuous Posets}
\date{\today}
\author{Jérémie Marquès}

\begin{document}
\maketitle

\begin{abstract}
	We organize and generalize several dualities involving continuous posets. The main theorem reads $\cSt_α\cInf_{α'}\cCont_{β'}\cSup_{β} ≃ (\cSt_β\cInf_{β'}\cCont_{α'}\cSup_{α})^\op$, where $\cSt$, $\cInf$, $\cCont$ and $\cSup$ refer to stability, completeness, continuity and cocompleteness. The indices are ``ladders,'' i.e., classes of sets $λ$ stable under dependent sums and quotients, with associated notions of $λ$-small infima and $λ$-filtered suprema. In the second half of the paper, we discuss algebraicity, proximity lattices and perfect maps.
\end{abstract}

\showcontents
\vspace{1em}

The goal of this paper is to organize several dualities that find their origin mainly in \cite{Law1979}. What started as an uninspired bookkeeping task turned out to reveal a nice symmetry between completeness and continuity on one hand, and between stability and cocompleteness on the other hand. For the sake of illustration, the reader can glance at the following list of dualities:
\begin{itemize}
	✦ Continuous posets are self-dual under Lawson duality. \cite{Law1979}
	✦ Continuous posets are dual to completely distributive frames. \cite{Law1979,Hof1981}, \cite[Thm.~7.21]{GehvGoo2024}
	✦ Completely distributive lattices are self-dual, constructively. \cite{RosWoo1994}
	✦ Sup-lattices are self-dual, constructively.
	✦ ``The Lawson duals of continuous preframes with bottom are precisely the continuous preframes with a compact top element.'' \cite[Prop.~I-1.8.2]{Kli2012}, \cite[Prop.~9.5]{Law1979}
	✦ ``The continuous complete Sup-lattices\footnote{``Sup-lattice'' here means that \emph{non-empty} suprema exist.} are precisely the Lawson duals of stably locally continuous domains [with a top element].'' \cite[Rmk.~after~Prop.~I-1.5.5]{Kli2012} \cite[Prop.~9.4]{Law1979}
	✦ The continuous meet-semilattices are self-dual. \cite[Thm.~7.4]{Law1979}
	✦ Stably continuous frames are self-dual under de Groot duality, constructively. \cite{Tow2006}
	✦ Stably completely distributive lattices are dual to cocomplete continuous posets, constructively.
\end{itemize}
It seems necessary to develop a uniform terminology for referring to these classes of posets. We will use \emph{ladders}%
\footnote{This terminology was suggested by André Joyal in the context of a different work. A ladder is what can be used to ``climb up'' a poset. We only consider here \emph{discrete} ladders, but there is an ordered notion which includes for instance the ladder of filtered posets.}
to parametrize them. A ladder is a class $α$ of sets which is closed under quotients and dependent sums. It is thus roughly the same thing as a regular cardinal, except that the inclusion of the empty set is independent from the rest, and that the totality of all sets is a ladder. There are associated notions of $α$-joins and $α$-filtered suprema.

The Main Theorem~\ref{thm:StLexContSup-dual} involves four ladders $α ⊆ α'$ and $β ⊆ β'$. It states that
\begin{equation*}
	\cSt_α\cInf_{α'}\cCont_{β'}\cSup_β ≃ (\cSt_β\cInf_{β'}\cCont_{α'}\cSup_{α})^\op
\end{equation*}
where the symbols $\cSt$, $\cInf$, $\cCont$, $\cSup$ refer respectively to stability, completeness, continuity and cocompleteness. The posets involved on both sides are supposed to ``have enough points'' in an appropriate sense. Using the terminology of Section~\ref{sec:dual}, $\cSt_α\cInf_{α'}\cCont_{β'}\cSup_β$ is the category of $α$-stably $β'$-continuous posets which are $α'$-complete, $β$-cocomplete and have enough $(α',β')$-points. The following conditions ensure that the requirement of having enough $(α',β')$-points is automatically satisfied:
\begin{itemize}
	✦ $α'$ consists only of finite sets; this is essentially the hypothesis taken in \cite{Law1979}.
	✦ $α = α'$; this has the advantage of working constructively.
	✦ $β'$ is the class of all sets; this also works constructively.
\end{itemize}
Theorem~\ref{thm:StLexContSup-dual} builds on Theorem~\ref{thm:FlatCont-dual}, stating that
\begin{equation*}
	\cFlat_α\cCont_β ≃ (\cFlat_β\cCont_α)^\op
\end{equation*}
where $\cFlat_α\cCont_β$ is the category of ``$β$-continuous posets with enough $(α,β)$-points.'' This duality relies on the interaction between a poset and its dual inside a shared \emph{$(α,β)$-intermediate structure} coming from the theory of polarities and canonical extensions \cite{FusGeh2025}.

In Section~\ref{sec:supp}, we adapt the theory of proximity lattices from \cite{JunSun1996} (see also \cite{Smy1992a,vanGool2012,Kaw2021}) and we briefly discuss perfect maps. We leave several topics untouched: distributivity laws, canonical extensions \cite{GehHar2001}, patch constructions, and the categorical structure of $\cSt_α\cSup_α$ (Is it star-autonomous? When is it compact closed?).

\paragraph{Related works and general comments}
The proofs and ideas involved here are all standard in lattice theory \cite{GieHofKeiLawMisSco2003}. As already mentioned, the main source of inspiration is the study of continuous posets in \cite{Law1979}. This paper, together with \cite{Hof1981a,Hof1981}, initiated further work on continuous posets \cite{Mar1981,HofMis1985,Law1987}.

Generalizations of continuous posets are studied under the name of $\Zc$-continuous posets \cite{WriWagTha1978,Nov1982,BanErn1983,Bar1996,Ern1999,ErnZha2001,YuaLi2019}, where $\Zc$ plays a role analogous to the ladders of the current paper. One difference is that we restrict here the notion of ``ladder'' to classes of sets, since $\Zc$-posets seem a priori too general for the goal of the current paper; this would, of course, be worth investigating. A theory of $α$-complete posets is developed in \cite[Ch.~XII]{BalDwi1974}, where $α$ is a cardinal. It mainly focuses on Boolean algebras and distributivity conditions. Coincidentally, we use the same letters $α$ and $β$ here, but they do not play the same role.

In \cite{Ern2016}, a choice-free version of Lawson's self-duality of continuous posets is developed. The core observation is that we can isolate the use of dependent choice in \cite{Law1979} into a simple condition that might not hold constructively. We use the same abstraction here (``having enough points''). Constructive validity was, in fact, important in the development of the current work because it provided a criterion to isolate the most fundamental aspects. The first step was to generalize de Groot duality, which resulted into Theorem~\ref{thm:StContSup-dual}. It was then noticed that the reasoning carries over into a more general context if we abstract away the condition on points. However, I decided to not emphasize the constructive aspects because I believe this is, at my current stage of understanding, gratuitous generality (see Remark~\ref{rmk:ladder}).

The idea of ``compact-open dualities'' is widely explored in the literature. For instance, the ``bi-dcpo's'' of the recent work \cite{AbbJun2026} play a role similar to the ``$(α,β)$-intermediate structures'' we use here. It is explained in \cite{JunSun1996} (see also \cite{Smy1992a}) how to present stably continuous frames using \emph{proximity lattices} while respecting the ``compact-open duality.'' The relation with canonical extensions is explored in \cite{vanGool2012}. See also \cite{Kaw2021} and the references therein for more information about proximity lattices and their history. The symmetry between compact and open is fundamental to the theory of canonical extensions \cite{GehVos2011,FusGeh2025}. The interested reader can consult the literature on bitopological spaces, d-frames and similar notions \cite{Kel1963,Kop1995,Jung06,Kli2012,Jak17,Sua2022}. The compact-open duality in topology specializes the proper-étale duality in topos theory, where the perfect symmetry between the two sides is however lost \cite{Ver1994,MoeVer2000,Tow2006}.

\paragraph{Conventions}
When $α$ is a ladder (Definition~\ref{dfn:ladder}), we use the symbols $\sjoin{α}$, $\fjoin{α}$, $\smeet{α}$, $\fmeet{α}$ to denote respectively $α$-joins, $α$-filtered suprema, $α$-meets and $α$-filtered infima. The terminology ``co-filtered'' will not be used. By a sup-lattice, we mean a poset with all suprema. Given a poset $P$, we denote by $\Down(P)$ the poset of down-sets of $P$ ordered by inclusion, i.e., the free sup-lattice on $P$. Dually, $\Up(P)^\op$ is the free inf-lattice on $P$. If $P$ is a poset, we denote by ${↓}p$ the down-set generated by $p$ and by ${↑}p$ the up-set generated by $p$. This identifies canonically $P$ with a subset of $\Down(P)$, and also with a subset of $\Up(P)^\op$. By a \emph{poset map} we mean an order-preserving map. We denote by $\fix(e)$ the set of fixpoints of an idempotent $e$.

\paragraph{Acknowledgments} I thank Mai Gehrke and Sam van Gool for their valuable comments on this paper. This work has received funding from the European Research Council under the European Union's Horizon 2020 research and innovation programme (Synergy Project Malinca, ERC Grant Agreement No 670624).

\section{Polarities and their intermediate structures}

Polarities are a way of presenting complete posets by means of a relation between two sets. It will be useful here to consider an ordered version. We review some basic facts and we fix some conventions. We refer to \cite[Sec.~3.1]{FusGeh2025} for more details.

In general, every algebraic structure $A$ is the quotient of a free structure. This can be used to present $A$ by generators and relations. Complete posets can be viewed as algebraic structures in two ways: either as sup-lattices or as inf-lattices. If $L$ is a complete lattice, it can therefore be written both as a quotient of a free sup-lattice, and as a quotient of a free inf-lattice:
\[\begin{tikzcd}
	\Down(X) \ar[r,"\text{sup-preserving}",->>,"p"'] &[5em] L &[5em] \Up(Y)^\op \ar[l,"\text{inf-preserving}"',->>,"q"]
\end{tikzcd}\]
By taking the left adjoint of $q$, we obtain an image-factorization diagram.
\[\begin{tikzcd}
	\Down(X) \ar[r,"\text{sup-preserving}",->>] &[4em] L \ar[r,"\text{sup-preserving}",>->] &[4em] \Up(Y)^\op
\end{tikzcd}\]
Hence any sup-lattice $L$ can be written as the image of some sup-preserving map $f : \Down(X) → \Up(Y)^\op$. Such a map $f$ is uniquely determined by an up-set $R ⊆ X^\op × Y$, verifying $R(x,y) ⇔ f(x)≤y$, using the canonical identifications $X ⊆ \Down(X)$ and $Y ⊆ \Up(Y)^\op$. In terms of the maps $p$ and $q$, we have $R(x,y) ⇔ p(x)≤q(y)$. We call $R$ a \emph{polarity} and we say that the image of $f$ is the complete lattice \emph{presented} by $R$.\footnote{The complete lattice presented by $R$ is often defined equivalently as the poset of fixpoints of the adjunction between $f$ and its right adjoint. See \cite[Prop.~1.39]{FusGeh2025}.}

Given a poset $X$ and a complete lattice $L$ we say that a poset map $X→L$ is \emph{join-dense} if the induced map $\Down(X)→L$ is surjective. Meet-density is defined dually. From the discussion above, the complete lattice $L$ presented by $R ⊆ X^\op×Y$ is uniquely determined by the following conditions \cite[Thm.~3.2]{FusGeh2025}:
\begin{itemize}
	✦ There is a join-dense poset map $p : X→L$.
	✦ There is a meet-dense poset map $q : Y→L$.
	✦ $p(x) ≤ q(y) ⇔ R(x,y)$.
\end{itemize}

\begin{exmp}{}{KOrd}
	Let $X$ be a compact ordered space, i.e., a compact Hausdorff space equipped with a closed order relation. Let $\Oc$ be the poset of open up-set of $X$ and let $\Kc$ be the poset of closed up-sets of $X$. Let $R ⊆ \Kc^\op × \Oc$ be the containment relation: $R(A,U) ⇔ [A⊆U]$. Then $R$ is a polarity which presents the complete lattice of all up-sets of $X$. This example should be kept in mind during the whole paper. It is fundamental to the theory of canonical extensions.
\end{exmp}

Based on the example above, we will write polarities as pairs $(\Oc,\Kc)$. The relation is omitted and is simply denoted by the symbol $≤$.

\paragraph{The intermediate structure} Let $L$ be the complete lattice presented by a polarity $(\Oc,\Kc)$. The \emph{intermediate structure} of $(\Oc,\Kc)$ is the disjoint union $\Oc∪\Kc$ equipped with the preorder induced by the map $\Oc∪\Kc → L$. This preorder is called the \emph{polarity order}, to distinguish it from the original orders of $\Oc$ and $\Kc$. It can be characterized as follows, where $U,V∈\Oc$ and $A,B∈\Kc$ \cite[Prop.~3.4]{FusGeh2025}:
\begin{itemize}
	✦ $A≤U$ is specified directly by the polarity.
	✦ $U≤V$ iff $∀A∈\Kc : A≤U ⇒ A≤V$.
	✦ $A≤B$ iff $∀U∈\Oc : B≤U ⇒ A≤U$.
	✦ $U ≤ A$ iff $∀B∈\Kc,V∈\Oc : (B≤U) ∧ (A≤V) ⇒ B≤V$.
\end{itemize}

\begin{rmk}{}{}
	The name ``intermediate structure'' comes from its use in \cite{GhiMel1997} as an ``intermediate level'' between syntax and semantics. The relevance of this idea was later rediscovered in \cite{DunGehPal2005} in the context of canonical extensions. See also \cite{GehJanPal2013}. (I thank Mai Gehrke for the explanation.)
\end{rmk}

\paragraph{Separating polarities} When the polarity order coincides on $\Oc$ and $\Kc$ with the original orders, we say that the polarity is \emph{separating}. Note that we could still have $A=U$ in the polarity order with $A ∈ \Kc$ and $U ∈ \Oc$.

\begin{exmp}{}{}
	Let $\Oc$ be a frame and let $X$ be a set of points of $\Oc$, ordered by the specialization order. There is a polarity $(\Oc,X^\op)$ in which $[x≤U] ⇔ x∈U$. This polarity is separating if and only if $X$ is a separating set of points of $\Oc$.
\end{exmp}

\begin{lem}{}{polarity-sep-inf}
	Let $(\Oc,\Kc)$ be a separating polarity presenting a complete lattice $L$. Then the map $\Kc → L$ preserves all infima that exist in $\Kc$.
\end{lem}

\begin{proof*}{}
	Since $(\Oc,\Kc)$ is separating, the map $\Kc → L$ is an order-embedding. It is also join-dense. Let $⋀_i A_i$ be an infimum in $\Kc$. Let $⋁_j B_j ∈ L$ with $(B_j)_j ⊆ \Kc$. Then
	\[ ⋁_j B_j ≤ ⋀_i A_i ⇔ ∀j,i : B_j ≤ A_i ⇔ ∀i : ⋁_j B_j ≤ A_i \text{.} \qedhere \]
\end{proof*}

\section{$(α,β)$-duality}
\label{sec:dual}

\subsection{Continuous posets}

\begin{dfn}{}{ladder}
	A (discrete) \emph{ladder} is a class of sets $α$ containing the singletons and that is closed under quotients and dependent sums. More explicitly:
	\begin{itemize}
		✦ If $I ∈ α$ and $I↠J$ is a surjection, then $J ∈ α$.
		✦ If $I ∈ α$ and $(J_i)_{i∈I} ⊆ α$, then $∑_{i∈I} J_i ∈ α$.
	\end{itemize}
\end{dfn}

\begin{exmp}{}{ladder}
	If $κ$ is a regular cardinal, we denote by $κ$ the ladder of sets of cardinality strictly less than $κ$. We denote by $κ^*$ the ladder of nonempty sets of cardinality strictly less than $κ$. We denote by $∞$ the ladder of all sets and by $∞^*$ the ladder of nonempty sets. We denote by $0$ the minimal ladder, containing only the singletons, and by $0^+$ the ladder containing the singletons and the empty set.
\end{exmp}

\begin{rmk}{}{ladder}
	In a sheaf topos, a ladder $α$ is instead encoded by the class of maps with $α$-small fibers. The axioms are:
	\begin{itemize}
		✦ $α$ is stable under pullbacks, sums, and for each pullback of the form below with $f$ surjective, $f^*(g) ∈ α ⇔ g ∈ α$ (in topos-theoretic language, $α$ is a stack).
		\[\begin{tikzcd}
			f^*(X) \ar[r] \ar[d,"f^*(g)"'] & X \ar[d,"g"]\\
			A \ar[r,->>,"f"] & B
		\end{tikzcd}\]
		✦ $α$ is stable under composition and contains the identities.
		✦ Let $f : A↠B$ be a surjection and let $g : B→C$. If $fg ∈ α$, then $g ∈ α$.
	\end{itemize}
	This definition coincides with Definition~\ref{dfn:ladder} in the case of the topos of sets: to a ladder $α$ in the sense of Definition~\ref{dfn:ladder}, we associate the class $\ovl{α}$ of all maps $f : X→Y$ such that $f^{-1}(y) ∈ α$ for all $y ∈ Y$. The stability of $α$ under dependent sums corresponds to the stability of $\ovl{α}$ under composition.
	
	The proofs below often require that every $α$-small object is (internally) projective. The only example I have in mind would be the ladder of cardinal finite objects, thus we will not pursue here the theory in such generality.
\end{rmk}

Let $α$ be a ladder. A set is \emph{$α$-small} when it is in $α$. An \emph{$α$-join} in a poset, written $\sjoin{α}_i x_i$, is a supremum indexed by an $α$-small set. An $α$-meet $\smeet{α}_i x_i$ is defined similarly. A poset is \emph{$α$-filtered} if every $α$-small subset has an upper bound. An \emph{$α$-filtered supremum} in a poset $P$ is a supremum $\fjoin{α}_{i∈I} x_i$ where $I$ is an $α$-filtered poset and $(i ∈ I) ↦ x_i ∈ P$ is order-preserving. Dually, an \emph{$α$-filtered infimum}, written $\fmeet{α}_i x_i$, is an $α$-filtered supremum in $P^\op$.

A down-set of a poset $P$ is an \emph{$α$-ideal} if it is $α$-filtered. We define dually $α$-filters. We denote by $\Idl_α(P)$ the poset of $α$-ideals of $P$. It is the free poset with $α$-filtered suprema over $P$. Dually, we denote by $\Flt_α(P)$ the poset of $α$-filters of $P$. We will often use the notation $\Flt_α^\op(P)$ to denote the poset of $α$-filters of $P$ ordered by reverse inclusion.

\begin{dfn}{}{}
	A poset $P$ is \emph{$α$-continuous} if there is a string of adjoints
	\[ {↡}_α(—) ⊣ \fjoin{α}(—) ⊣ {↓}(—) \]
	where ${↓}(—) : P→\Idl_α(P)$ is the map sending $x$ to the principal down-set ${↓}x$.
\end{dfn}

If $P$ is $α$-continuous, the relation $x ≪_α y$ is defined as $x ∈ {↡}_αy$. Equivalently, any $α$-filtered supremum above $y$ contains a term above $x$. This definition works more generally in any poset with $α$-filtered suprema, and $P$ is $α$-continuous if and only if every $x∈P$ is the $α$-filtered supremum of $\setst{y}{y ≪_α x}$. 

We say that a binary relation $R$ on a poset $P$ is \emph{order-compatible} if $[≤] ∘ R ∘ [≤] = R$, and that it is \emph{idempotent} if $R∘R = R$. The relation $≪_α$ on an $α$-continuous poset is order-compatible and idempotent.

A down-set of a poset $P$ with $α$-filtered suprema is \emph{$α$-closed} if it is closed under $α$-filtered suprema. Its complement is said to be \emph{$α$-open}. These subsets are stable under arbitrary unions. If $P$ is $α$-continuous, the $α$-open subsets are the fixpoints of the infosys \cite{Vic1993} determined by $≪_α$. The following lemma is also standard.

\begin{lem}{}{charact-open}
	In an $α$-continuous poset, an up-set $F$ is $α$-open iff
	\[ ∀x ∈ F : ∃ y ∈ F : y ≪_α x \text{.} \]
\end{lem}

Equivalently, the $α$-open subsets are the unions of subsets of the form ${↟}_αy$. Dually, in a poset $P$ with $α$-filtered infima, an \emph{$α$-co-open} is a down-set whose complement is closed under $α$-filtered infima.

\subsection{The $(α,β)$-intermediate structure of a $β$-continuous poset}

Let $α$ and $β$ be two ladders. Let $\Oc$ be a $β$-continuous poset. An \emph{$(α,β)$-point} of $\Oc$ is a $β$-open $α$-filter. If $\Oc$ has $α$-meets, it is thus the same thing as a map from $\Oc$ to the two-element lattice preserving $α$-meets and $β$-filtered joins. We denote by $\pt_{αβ}(\Oc)$ the poset of $(α,β)$-points of $\Oc$, ordered by inclusion.

\begin{exmp}{}{}
	If $\Oc$ is a continuous poset, i.e., an $ω$-continuous poset, then an $(ω,ω)$-point of $\Oc$ is a Scott-open filter. The Lawson dual \cite{Law1979} of $\Oc$ is thus $\pt_{ωω}(\Oc)$.
\end{exmp}


Fix two ladders $α$ and $β$. Let $\Oc$ be a $β$-continuous poset. We will use $\Kc$ to denote $\pt_{αβ}(\Oc)$ with the opposite order. We will systematically use the variables $U,V,W$ to denote elements of $\Oc$, and the variables $A,B,C$ to denote elements of $\Kc$, without necessarily mentioning for instance $A∈\Kc$. We define the polarity $(\Oc,\Kc) = (\Oc,\pt_{αβ}^\op(\Oc))$ by the relation $A ≤ U ⇔ U∈A$. We will from now on work with the associated polarity order.

The reader should keep in mind Example~\ref{exmp:KOrd} where $\Oc$ is the lattice of open up-sets of some compact ordered space, and that $\Kc$ is the lattice of closed up-sets of the same space. Then $A≤U$ means that $A$ is included in $U$.

\paragraph{Having enough points} The most obvious definition of ``$\Oc$ has enough $(α,β)$-points'' is probably to require that ``the points determine the order'' in the sense that $(\Oc,\Kc)$ is separating, in analogy with the case of frames. We will require something stronger, namely that ``the points determine the $≪_β$ relation.''

We write temporarily $≤^\oo$ for the original order on $\Oc$, and we write $U ≤^\oo A$ instead of $∀V : A≤V ⇒ U≤^\oo V$. If $(\Oc,\Kc)$ is separating, then these relations coincide with the polarity order $≤$.

\begin{dfn}{}{enough-pts}
	We say that a $β$-continuous poset $\Oc$ has \emph{enough $(α,β)$-points} if for all $U ≪_β V$ in $\Oc$, there is $A ∈ \Kc$ such that $U ≤^\oo A ≤ V$.
\end{dfn}

More explicitly, it means that if $U≪_βV$, then we can find some $β$-open $α$-filter $F$ such that $V ∈ F ⊆ {↑}U$. The converse of Definition~\ref{dfn:enough-pts} is always satisfied: if $U ≤^\oo A ≤ V$, then $U ≪_β V$.

\begin{lem}{}{enough-points-impl-sep}
	If $\Oc$ has enough $(α,β)$-points, then $(\Oc,\Kc)$ is separating.
\end{lem}

\begin{proof*}{}
	The restriction of the polarity order on $\Kc$ is trivially the original order. Let $U, V ∈ \Oc$ such that $∀A : A≤U ⇒ A≤V$. For all $W≪_βU$ there is $A$ such that $W≤^\oo A≤U$, hence $W≤^\oo A≤V$, hence $W≪_βV$. This shows that $U=⋁\setst{W}{W≪_βU}≤^\oo V$.
\end{proof*}

Thanks to Lemma~\ref{lem:enough-points-impl-sep}, we will not need to use the symbol $≤^\oo$ in the rest of this paper.
We recall the standard argument showing that $\Oc$ has enough points when $α$ contains only finite sets \cite[Prop.~2.2]{Law1979}.

\begin{prop}{}{}
	If $α$ contains only finite sets or $β=∞$, then $\Oc$ has enough $(α,β)$-points.
\end{prop}

\begin{proof*}{}
	Suppose that $α$ contains only finite sets. Suppose that $U ≪_β V$. We apply inductively that $≪_β$ is interpolative to obtain $U ≪_β ⋯ ≪_β W_2 ≪_β W_1 ≪_β V$. Then $⋃_i {↑}W_i = ⋃_i {↟}_βW_i$ is a $β$-open $ω$-filter, hence a $β$-open $α$-filter.
	
	Suppose that $β=∞$. Then $≪_β$ coincides with the order $≤$ and ${↑}U$ is a $β$-open $α$-filter for all $U$.
\end{proof*}

\paragraph{$(α,β)$-intermediate structures} We now characterize the polarities $(\Oc,\pt^\op_{αβ}(\Oc))$ where $\Oc$ has enough $(α,β)$-points.

\begin{dfn}{}{}
	An \emph{$(α,β)$-intermediate structure} is (the intermediate structure of) a polarity $(\Oc,\Kc)$ satisfying the following axioms:
	\begin{enumerate}[label=(IS\arabic*)]
		✦ The polarity is separating. \label{ax:IS-sep}
		✦ $\Oc$ has $β$-filtered suprema and $\Kc$ has $α$-filtered infima. \label{ax:filt-sup-inf}
		✦ For all $A ∈ \Kc$, the set $\setst{U∈\Oc}{A≤U}$ is a $β$-open $α$-filter. \label{ax:IS-K}
		✦ For all $U ∈ \Oc$, the set $\setst{A∈\Kc}{A≤U}$ is an $α$-co-open $β$-ideal. \label{ax:IS-O}
		✦ Every $A ≤ U$ can be completed to $A ≤ V ≤ B ≤ U$. \label{ax:IS-inter}
	\end{enumerate}
\end{dfn}

\begin{prop}{}{inter-to-cont}
	If $(\Oc,\Kc)$ is an $(α,β)$-intermediate structure, then:
	\begin{itemize}
		✦ $\Oc$ is $β$-continuous and $\Kc$ is $α$-continuous.
		✦ $U ≪_β V ⇔ ∃A : U≤A≤V$ and $A ≪_α B ⇔ ∃U : A≤U≤B$. 
		✦ $\Oc ≃ \pt_{βα}(\Kc^\op)$ and $\Kc^\op ≃ \pt_{αβ}(\Oc)$.
	\end{itemize}
\end{prop}

In a way, $\Kc$ witnesses the $β$-continuity of $\Oc$ and, conversely, $\Oc$ witnesses the $α$-continuity of $\Kc^\op$.

\begin{proof*}{}
	We only show the statements concerning $\Kc$.
	
	We show that if $A ≤ U ≤ B$ then $A ≪_α B$. Suppose that $A≤U≤B$ and that $\fmeet{α}_i A_i ≤ A$. Then $\fmeet{α}_i A_i ≤ U$ and by \ref{ax:IS-O} there is $i$ such that $A_i ≤ U ≤ B$. We now show that $A = \fmeet{α} \setst{B}{∃U : A≤U≤B}$ in $\Kc$. By definition of the polarity order, we have $A = ⋀\setst{U}{A≤U}$. Thus
	\[ A = ⋀\setst{B}{∃V,U : A≤V≤B≤U} = ⋀\setst{B}{∃V : A≤V≤B} \text{.} \]
	We show that this infimum is $α$-filtered. If $(A≤V_i≤B_i)_{i∈I}$ where $I$ is $α$-small then, by \ref{ax:IS-K}, there is $V$ such that $A≤V$ and $∀i : V≤V_i$. By \ref{ax:IS-inter}, there are $W$ and $B$ such that $A≤W≤B≤V$. Thus the infimum is $α$-filtered. Since moreover $A≤V≤B ⇒ A≪_αB$, this shows that $\Kc$ is $α$-continuous and that $A≪_αB ⇔ ∃U : A≤U≤B$. Dually, $\Oc$ is $β$-continuous and the dual characterization of $≪_β$ holds.
	
	Finally, we show that $\Kc^\op ≃ \pt_{αβ}(\Oc)$. By \ref{ax:IS-K}, there is a map $\Kc^\op → \pt_{αβ}(\Oc)$ sending $A$ to $\setst{U}{A≤U}$. It is an order-embedding by \ref{ax:IS-sep}. We show it is surjective. Let $F ⊆ \Oc$ be a $β$-open $α$-filter. Let $F^* ≔ \setst{B}{∃U∈F : U≤B}$. We show that $F^*$ is $α$-filtered. If $(U_i ≤ B_i)_{i∈I}$ for some $α$-small $I$, then since $F$ is an $α$-filter there is $U$ such that $∀i : U≤U_i≤B_i$. Since $F$ is $β$-open, there is $U'∈F$ such that $U'≪_βU$. By the characterization of $≪_β$, we get $B$ with $U'≤B≤U$ so that $B ∈ F^*$ and $B≤B_i$ for all $i$. We have shown that $F^*$ is $α$-filtered. Let $A ≔ \fmeet{α} F^* ∈ \Kc$. We show that $U ∈ F ⇔ A≤U$. If $U∈F$, then there is $F∋U'≪_βU$ hence $U'≤B≤U$ and $A = \fmeet{α} F^* ≤ B ≤ U$. Reciprocally, suppose that $A≤U$. Then by \ref{ax:IS-O} there is $U'≤B≤U$ with $U' ∈ F$, hence $U ∈ F$.
\end{proof*}

\begin{prop}{}{cont-to-inter}
	If $\Oc$ is $β$-continuous and has enough $(α,β)$-points, then $(\Oc,\pt_{αβ}^\op(\Oc))$ is an $(α,β)$-intermediate structure.
\end{prop}

\begin{proof*}{}
	Axiom \ref{ax:IS-sep} holds by Lemma~\ref{lem:enough-points-impl-sep}.
	Axiom \ref{ax:filt-sup-inf} says that $\pt_{αβ}(\Oc)$ has $α$-filtered suprema. This holds because an arbitrary union of $β$-open subsets is open, and an $α$-filtered union of $α$-filters is an $α$-filter.
	Axiom \ref{ax:IS-K} holds by definition of $\pt_{αβ}(\Oc)$. Axiom \ref{ax:IS-inter} holds because if $A ≤ U$, then there is $A ≤ V ≪_β U$ and thus $A ≤ V ≤ B ≤ U$ since $\Oc$ has enough $(α,β)$-points.
	Finally, we check \ref{ax:IS-O}: Let $U ∈ \Oc$ and let $I ≔ \setst{A∈\pt_{αβ}^\op(\Oc)}{A≤U}$. As we saw above, $α$-filtered suprema in $\pt_{αβ}(\Oc)$ are computed as unions, hence $I$ is $α$-co-open. We show that it is a $β$-ideal: Suppose that $(A_i ≤ U)_i$ is a $β$-small family. For each $i$, let $A_i ≤ U_i ≪_β U$. Since $\Oc$ is $β$-continuous, there is $V ≪_β U$ such that $∀i : U_i ≪_β V$. Since $\Oc$ has enough $(α,β)$-points, we get $V ≤ A ≤ U$. Thus $I$ is a $β$-ideal.
\end{proof*}

\begin{cor}{}{}
	Let $\Oc$ be a $β$-continuous lattice with enough $(α,β)$-points. Then the evaluation map
	\[ \Oc \longrightarrow \pt_{βα}(\pt_{αβ}(\Oc)) \]
	is an isomorphism.
\end{cor}

In this situation, we say that $\pt_{αβ}(\Oc)$ is the \emph{$(α,β)$-dual} of $\Oc$.

\subsection{Basic $(α,β)$-duality}

\begin{dfn}{}{}
	An \emph{$(α,β)$-flat map} $f : \Oc_1 → \Oc_2$ between two $β$-continuous posets is a poset map whose inverse image $f^{-1}$ preserves $β$-open $α$-filters.
\end{dfn}

We denote by $\cFlat_α\cCont_β$ the category of $β$-continuous posets with enough $(α,β)$-points, and $(α,β)$-flat maps. It is not obvious that the notion of $(α,β)$-flatness makes sense in general. It is the minimal condition necessary to obtain a dual map $\pt_{αβ}(\Oc_2) → \pt_{αβ}(\Oc_1)$. We illustrate this with two special cases:
\begin{itemize}
	✦ $f$ is $(0,β)$-flat exactly when it preserves $β$-filtered suprema. This generalizes the fact that the Scott-continuous maps are those preserving filtered suprema.
	✦ $f$ is $(α,∞)$-flat exactly when it is representably $α$-flat in the categorical sense.
\end{itemize}

We will see in Proposition~\ref{prop:ex-full-flat} that, under the assumption that $\Oc_1$ has $α$-meets and that $\Oc_2$ has enough $(α,β)$-points, the $(α,β)$-flat maps coincide with the poset maps preserving $α$-meets and $β$-filtered suprema. 
More generally, when we just suppose that $\Oc_2$ has enough $(α,β)$-points, $(α,β)$-flatness can be understood as a sort of $≪_β$-relaxed representable $α$-flatness (Remark~\ref{rmk:charact-flat}).

\begin{dfn}{}{}
	A \emph{morphism} $f : (\Oc_1,\Kc_1) → (\Oc_2,\Kc_2)$ between two $(α,β)$-intermediate structures is an adjoint pair
	\[\begin{tikzcd}[row sep=0.6em]
		\Oc_1 \ar[r,"f^*",""{below,name=A}] & \Oc_2 \\
		\Kc_1 & \Kc_2 \ar[l,"f_!",""{above,name=B}]
		\ar[from=A,to=B,"⊢"{description,sloped},phantom]
	\end{tikzcd}\]
	in the sense that $f_!(A) ≤ U ⇔ A ≤ f^*(U)$.
\end{dfn}

We denote by $\cInter_{αβ}$ the category of $(α,β)$-intermediate structures. It bears an obvious self-duality that swaps $\Oc$ and $\Kc$ while reversing the order.

\begin{exmps}{}{}~
	\begin{itemize}
		✦ With $α=β=ω$, suppose that $(\Oc_1,\Kc_1)$ and $(\Oc_2,\Kc_2)$ are obtained from compact ordered spaces $X_1$ and $X_2$ as in Example~\ref{exmp:KOrd}. If $f^*$ is the inverse image along a morphism $f : X_2→X_1$, then the map $f_!$ sends a closed up-set $A ⊆ X_2$ to ${↑}f[A]$.
		✦ When $α=β=∞$, an $∞$-continuous poset $\Oc$ is just a poset and $\pt_{∞∞}(\Oc)$ is the poset of principal up-sets of $\Oc$, i.e., $\Oc^\op$. In this case $\Oc = \Kc$ and a morphism of $(∞,∞)$-intermediate structures is just an adjoint pair. When $\Oc$ is complete, $f^*$ is an inf-preserving map and $f_!$ is its right adjoint.
	\end{itemize}
\end{exmps}

\begin{thm}{}{FlatCont-dual}
	\[ \cFlat_α\cCont_β ≃ \cInter_{αβ} ≃ \cInter_{βα}^\op ≃ (\cFlat_β\cCont_α)^\op \]
\end{thm}

\begin{proof*}{}
	Propositions~\ref{prop:inter-to-cont} and \ref{prop:cont-to-inter} tell us that the objects of the two categories $\cFlat_α\cCont_β$ and $\cInter_{αβ}$ correspond to each other. The equivalence at the level of arrows is tautological: By definition, a map $f^* : \Oc_1 → \Oc_2$ is $(α,β)$-flat iff there is a map $f_! : \pt_{αβ}^\op(\Oc_2) → \pt_{αβ}^\op(\Oc_1)$ such that $\setst{U}{f_!(A)≤U} = (f^*)^{-1}(\setst{V}{A≤V})$, i.e., such that $f_!(A) ≤ U ⇔ A ≤ f^*(U)$.
\end{proof*}

\begin{prop}{}{flat-impl-cont}
	If $f^* : \Oc_1 → \Oc_2$ is $(α,β)$-flat and if $\Oc_2$ has enough $(α,β)$-points, then $f^*$ preserves $β$-filtered suprema.
\end{prop}

\begin{proof*}{}
	\begin{align*}
		A ≤ f^*\p[\Big]{\fjoin{β}_i U_i} &⇔ f_!(A) ≤ \fjoin{β}_i U_i \\
		&⇔ ∃i : f_!(A) ≤ U_i \\
		&⇔ ∃i : A ≤ f^*(U_i) \\
		&⇔ A ≤ \fjoin{β}_i f^*(U_i) \qedhere
	\end{align*}
\end{proof*}

\begin{rmk}{}{}
	The proof of Theorem~\ref{thm:FlatCont-dual} is similar to the standard proof showing that left adjoints preserve colimits. See also the proof of Proposition~\ref{prop:ex-full-flat}, where the preservation of $α$-infima is treated similarly.
\end{rmk}

\begin{rmk}{}{charact-flat}
	If $\Oc_2$ has enough $(α,β)$-points, then a poset map $f^* : \Oc_1 → \Oc_2$ is $(α,β)$-flat if and only if it preserves $β$-filtered suprema and:
	\begin{center}
		For all $α$-small $(U_i)_i ⊆ \Oc_1$ and for all $V,V' ∈ \Oc_2$,\\
		if $∀i : V ≤ f^*(U_i)$ and $V' ≪_β V$, then\\
		there is $U ∈ \Oc_1$ such that $∀i : U ≤ U_i$ and $V' ≤ f^*(U)$.
	\end{center}
	\[\begin{tikzcd}
		U_i \ar[rr,mapsto] &[2em] &[-2em] f^*(U_i) \\[-1.5em]
		∃U \ar[r,mapsto,dashed] \ar[u,dashed,-] & f^*(U) \ar[ru,dashed,-] & V \ar[u,-] \\[-2em]
		& & V' \ar[u,"≪"{sloped,description},phantom] \ar[lu,dashed,-]
	\end{tikzcd}\]
\end{rmk}

\subsection{Exactness}

We now replace flatness by exactness. We denote by $\cInf_α\cCont_β$ the category of $α$-complete $β$-continuous posets with enough $(α,β)$-points, and poset maps preserving $α$-meets and $β$-filtered suprema. We will see that it is a full subcategory of $\cFlat_α\cCont_β$ and that $\cInf_α\cCont_β ≃ (\cInf_β\cCont_α)^\op$.

\begin{prop}{}{meet-corresp}
	Let $(\Oc,\Kc)$ be an $(α,β)$-intermediate structure. Then $\Oc$ has $α$-meets iff $\Kc$ has $β$-joins.
\end{prop}

\begin{proof*}{}
	Suppose that $\Oc$ has $α$-meets. Then $α$-filters are stable under arbitrary intersections. Consequently, $β$-open $α$-filters of $\Oc$ are stable under $β$-small intersections. This shows one half of the proposition and the other one is dual.
\end{proof*}

\begin{dfn}{}{}
	We say that an $(α,β)$-intermediate structure $(\Oc,\Kc)$ is \emph{exact} when $\Oc$ has $α$-meets or, equivalently, when $\Kc$ has $β$-joins.
\end{dfn}

In this case, $\Oc$ is stable under $α$-meets in the intermediate structure, and $\Kc$ is stable under $β$-joins.

\begin{prop}{}{ex-full-flat}
	$\cInf_α\cCont_β$ is a full subcategory of $\cFlat_α\cCont_β$.
\end{prop}

\begin{proof*}{}
	Let $\Oc_1 ∈ \cInf_α\cCont_β$ and let $\Oc_2 ∈ \cFlat_α\cCont_β$. If $f : \Oc_1 → \Oc_2$ preserves $β$-filtered suprema and $α$-infima, then $f^{-1}$ preserves $β$-open $α$-filters because it preserves both $β$-open subsets and $α$-filters, independently. Reciprocally, let $f^* : \Oc_1 → \Oc_2$ be an $(α,β)$-flat map, with dual $f_! : \Kc_2 → \Kc_1$. By Proposition~\ref{prop:flat-impl-cont}, $f^*$ preserves $β$-filtered suprema. We show that it also preserves $α$-infima:
	\begin{align*}
		A ≤ f^*\p[\Big]{\smeet{α}_i U_i} &⇔ f_!(A) ≤ \smeet{α}_i U_i\\
		&⇔ ∀i : f_!(A) ≤ U_i\\
		&⇔ ∀i : A ≤ f^*(U_i)\\
		&⇔ A ≤ \smeet{α}_i f^*(U_i) \qedhere
	\end{align*}
\end{proof*}

\begin{thm}{}{LexCont-dual}
	$\cInf_α\cCont_β ≃ (\cInf_β\cCont_α)^\op$
\end{thm}

\begin{proof*}{}
	By Proposition~\ref{prop:meet-corresp}, the $(α,β)$-dual of $\Oc ∈ \cInf_α\cCont_β$ is in $\cInf_β\cCont_α$. By Proposition~\ref{prop:ex-full-flat}, the duality $\cFlat_α\cCont_β ≃ (\cFlat_β\cCont_α)^\op$ restricts to $\cInf_α\cCont_β ≃ (\cInf_β\cCont_α)^\op$.
\end{proof*}

\subsection{Stability}

\begin{dfn}{}{}
	Let $\Oc$ be a $β$-continuous poset with $α$-meets. We say that it is \emph{$α$-stably} $β$-continuous if the map ${↡}_β : \Oc → \Idl_β(\Oc)$ preserves $α$-meets. Equivalently, if $(U ≪_β U_i)_{i∈I}$ for some $α$-small $I$, then $U ≪_β \smeet{α}_i U_i$.
\end{dfn}

If $α ⊆ α'$ and $β ⊆ β'$ are ladders, we let
\[ \cSt_α\cInf_{α'}\cCont_{β'}\cSup_{β} ⊆ \cInf_{α'}\cCont_{β'} \]
be the full sub-category spanned by the $α$-stably $β'$-continuous posets which are $α'$-complete and $β$-cocomplete.

\begin{prop}{}{StCont-corresp}
	Let $(\Oc,\Kc)$ be an exact $(α,β)$-intermediate structure. Let $γ ⊆ α$. Then $\Oc$ is $γ$-stably $β$-continuous if and only if $\Kc$ is $γ$-complete. Moreover, in this case
	\[ \smeet{γ}_i A_i ≤ U ⇔ ∃(A_i≤V_i)_i : \smeet{γ}_iV_i≤U \text{.} \]
\end{prop}

\begin{proof*}{}
	Suppose that $\Oc$ is $γ$-stable. Let $(F_i)_i$ be a $γ$-family of $β$-open $α$-filters of $\Oc$. We claim that its join is the $α$-filter generated by the meets of the form $\smeet{γ}_i U_i$ where $U_i ∈ F_i$ for all $i$. We check it is $β$-open: for every family $(U_i ∈ F_i)_i$, we have another family $(U_i ≫_β V_i ∈ F_i)_i$. Thus $\smeet{γ}_i V_i ≪_β \smeet{γ}_i U_i$ by $γ$-stability of $\Oc$. This shows that $\Kc^\op$ has $γ$-suprema.
	
	For the other direction, suppose that $\Kc$ is $γ$-complete. We show that $\Oc$ is $γ$-stably $β$-continuous. Suppose that $(U ≪_β U_i)_i$ is a $γ$-small family. Then there is $(U ≤ A_i ≤ U_i)_i$. Thus $U ≤ \smeet{γ}_i A_i$ because $\Kc$ has $γ$-meets and by Lemma~\ref{lem:polarity-sep-inf}. Since $γ ⊆ α$, we have $U ≤ \smeet{γ}_i A_i ≤ \smeet{γ}_i U_i$, hence $U ≪_β \smeet{γ}_i U_i$. This shows that $\Oc$ is $γ$-stably $β$-continuous.
\end{proof*}

We deduce the main duality of the paper.

\begin{thm}{}{StLexContSup-dual}
	Let $α ⊆ α'$ and $β ⊆ β'$ be ladders. Then
	\[ \cSt_α\cInf_{α'}\cCont_{β'}\cSup_{β} ≃ (\cSt_{β}\cInf_{β'}\cCont_{α'}\cSup_{α})^\op \]
\end{thm}

When $α=α'$, we simply write $\cSt_α$ instead of $\cSt_α\cInf_{α'}$. When $β=β'$, we simply write $\cSup_β$ instead of $\cCont_{β'}\cSup_β$. When $α=0$, we omit $\cSt_α$. When $β=0$, we omit $\cSup_β$. For instance, $\cSt_α\cCont_β$ is the category of $α$-stably $β$-continuous posets with enough $(α,β)$-points. The nice thing in this case is that the condition of having enough $(α,β)$-points is trivial, and even constructively so.

\begin{prop}{}{StCont-pointfull}
	Any $α$-stably $β$-continuous poset $\Oc$ has enough $(α,β)$-points.
\end{prop}

\begin{proof*}{}
	For all $U ∈ \Oc$, the subset ${↟}_β U$ itself is a $β$-open $α$-filter.
\end{proof*}

\begin{thm}{}{StCont-LexContSup-dual}
	$\cSt_α\cCont_β ≃ (\cInf_β\cSup_α)^\op$
\end{thm}

\begin{thm}{}{StContSup-dual}
	$\cSt_α\cSup_β ≃ (\cSt_β\cSup_α)^\op$
\end{thm}

\section{Supplements}
\label{sec:supp}

In this second part, we start by reproducing the analysis of \cite{JunSun1996} (see also \cite{vanGool2012}) and then we show how perfect maps fit in the picture.

\subsection{Algebraicity and proximity posets}

Let $\Oc$ be a $β$-continuous poset. We say that $U ∈ \Oc$ is \emph{$β$-compact} or \emph{$β$-algebraic} if $U ≪_β U$. We denote by $\Alg_β(\Oc)$ the sub-poset of $β$-algebraic elements of $\Oc$. It is the set of $U ∈ \Oc$ such that ${↑}U$ is $β$-open. It is thus the intersection $\Oc ∩ \Kc$ if $(\Oc,\Kc)$ is an $(α,β)$-intermediate structure, for any $α$.

\begin{dfn}{}{}
	A $β$-continuous poset $\Oc$ is \emph{$β$-algebraic} if it is of the form $\Idl_β(P)$.
\end{dfn}

The standard arguments apply. For instance, $\Oc$ is algebraic exactly when every element is the $β$-filtered supremum of the algebraic elements below it. We recover $P$ as $\Alg_β(\Idl_β(P))$. If $\Oc$ is algebraic, then it has enough $(α,β)$-points for any $α$, and $U ≪_β V$ iff there is $x ∈ \Alg_β(\Oc)$ such that $U ≤ x ≤ V$. The $(α,β)$-dual of $\Idl_β(P)$ is $\Flt_α(P)$. In the $(α,β)$-intermediate structure $(\Idl_β(P),\Flt_α^\op(P))$, $P$ is identified with $\Idl_β(P) ∩ \Flt_α^\op(P)$. For all $F ∈ \Flt_α^\op(P), I ∈ \Idl_β(P)$, we have
\[ F ≤ I ⇔ ∃ p∈P : F ≤ p ≤ I \text{.} \]

\begin{dfn}{}{}
	An $(α,β)$-intermediate structure is called \emph{algebraic} when it is of the form $(\Idl_β(P),\Flt_α^\op(P))$.
\end{dfn}

We can describe explicitly the category of algebraic $(α,β)$-intermediate structures. Given two posets $P$ and $Q$, a morphism $(\Idl_β(P),\Flt_α^\op(P)) → (\Idl_β(Q),\Flt_α^\op(Q))$ is a pair of adjoints $f : \Idl_β(P) → \Idl_β(Q)$ and $g : \Flt_α^\op(Q) → \Flt_α^\op(P)$. Since $f$ preserves $β$-filtered suprema and since $g$ preserves $α$-filtered infima, these maps are given by two relations $R ⊆ Q^\op×P$ and $S ⊆ Q^\op×P$. The adjunction relation moreover gives that $R=S$. We use the notation $R(—,p) ≔ \setst{q}{R(q,p)}$ and similarly for $R(q,—)$. From this discussion, the morphisms from $(\Idl_β(P),\Flt_α^\op(P))$ to $(\Idl_β(Q),\Flt_α^\op(Q))$ are in bijection with the up-sets $R ⊆ Q^\op × P$ such that $R(—,p)$ is a $β$-ideal for all $p∈P$ and such that $R(q,—)$ is an $α$-filter for all $q∈Q$. The composition of these relations is the usual composition of relations.

\begin{prop}{}{inter-cauchy-compl}
	The category $\cInter_{αβ}$ is the Cauchy completion of the category of algebraic $(α,β)$-intermediate structures.
\end{prop}

In fact, we will show that $\cInter_{αβ}$ is the completion of the category of algebraic $(α,β)$-intermediate structures under splitting of \emph{idempotent comonads}. It turns out to also be Cauchy complete.

\begin{proof*}{}
	Let $(\Oc,\Kc) ∈ \cInter_{αβ}$. We wish to show that it is the adjoint retract of some $(\Idl_β(P),\Flt_α^\op(P))$. Since $\Oc$ is $β$-continuous, we have the following adjoint retract.
	\[\begin{tikzcd}
		\Oc \ar[r,shift right=0.8em,"{↡}_β(—)"',""{above,name=B}] & \Idl_β(\Oc) \ar[l,shift right=0.8em,"∨(—)"',""{below,name=A}]
		\ar[from=A,to=B,phantom,"⊢"{sloped,description}]
	\end{tikzcd}\]
	However, there is no reason for it to be compatible with $\Kc$. Instead, we use the construction from \cite[Sec.~6]{JunSun1996}. Let $P = \setst{(U,A) ∈ \Oc×\Kc}{U≤A}$. We define the following maps:
	\begin{itemize}
		✦ $\Oc → \Idl_β(P)$ sends $V ∈ \Oc$ to $\setst{U≤A}{A≤V}$. It is indeed a $β$-ideal because if $(U_i ≤ A_i ≤ V)_i$ is a $β$-small family, then there is $A ≤ V$ with $∀i : A_i≤A$, and we obtain $A ≤ U ≤ B ≤ V$ by \ref{ax:IS-inter} so that $(U≤B) ∈ P$ is an upper bound.
		✦ $\Idl_β(P) → \Oc$ sends $\set{U_i ≤ A_i}_i$ to $\fjoin{β}_i U_i$.
		✦ The maps $\Kc → \Flt_α^\op(P)$ and $\Flt_α^\op(P) → \Kc$ are defined similarly.
	\end{itemize}
	We then check:
	\begin{itemize}
		✦ The composite $\Oc → \Idl_β(P) → \Oc$ is the identity: It is
		\[ (U ∈ \Oc) ↦ \fjoin{β} \setst{V}{∃A : V≤A≤U} = U \text{.} \]
		✦ The composite $\Idl_β(P) → \Oc → \Idl_β(P)$ is below the identity: It is
		\[ (\set{U_i≤A_i}_i ∈ \Idl_β(P)) ↦ \setst{U≤A}{A≤\fjoin{β}_i U_i} = \setst{U≤A}{∃i : A≤U_i} \text{.} \]
		✦ $\Oc → \Idl_β(P)$ is right adjoint to $\Flt_α^\op(P) → \Kc$: Let $U ∈ \Oc$ and let $\set{U_i≤A_i}_i ∈ \Flt_α^\op(P)$. Then
		\[ \fmeet{α}_i A_i ≤ U ⇔ ∃i : A_i ≤ U ⇔ \set{U_i≤A_i}_i ≤ \setst{V≤B}{B≤U} \text{.} \]
		✦ Dually, $\Idl_β(P) → \Oc$ is right adjoint to $\Kc → \Flt_α^\op(P)$.
	\end{itemize}
	This shows that $(\Oc,\Kc)$ is an adjoint retract of $(\Idl_β(P),\Flt_α^\op(P))$.
	
	We show now that $\cInter_{αβ}$ is Cauchy-complete. It would be more natural to show that it is closed under splitting of idempotent comonads, but this is a weaker statement. Let $(\Oc,\Kc) ∈ \cInter_{αβ}$. Let $e : \Oc → \Oc$ be idempotent and right adjoint to $f : \Kc → \Kc$. Let $\fix(e) ⊆ \Oc$ and $\fix(f) ⊆ \Kc$ be the sets of fixpoints of $e$ and $f$, with the induced order. We equip $(\fix(e),\fix(f))$ with the restricted polarity. We show that this produces an $(α,β)$-intermediate structure. In order to distinguish the polarity orders of $(\Oc,\Kc)$ and $(\fix(e),\fix(f))$, we shall write $≤^*$ for the latter. Note that $≤^*$ coincides with $≤$ by definition on inequalities of the form $A≤U$, $A≤B$ and $U≤V$. On the other hand, $U ≤^* A$ is equivalent to $∀B ∈ \fix(f) : B≤U ⇒ B≤A$, which is weaker than $U≤A$ (unless $(e,f)$ is an idempotent comonad). We now check the various axioms of $(α,β)$-intermediate structures:
	\begin{itemize}
		✦[\ref{ax:IS-sep}] The polarity is separating: This is simply because if $A ∈ \Kc$ and $U ∈ \fix(e)$, then $A ≤ U ⇔ A ≤ e(U) ⇔ f(A) ≤ U$.
		✦[\ref{ax:filt-sup-inf}] $\fix(e)$ has $β$-filtered suprema because $e$ preserves them.
		✦[\ref{ax:IS-K}] Let $A ∈ \fix(f)$. We check that $\setst{U∈\fix(e)}{A≤U}$ is a $β$-open $α$-filter. It is $β$-open because $β$-filtered suprema in $\fix(e)$ are computed in $\Oc$. It is an $α$-filter because if $(A ≤ U_i)_i$ is $α$-small with $U_i ∈ \fix(e)$, then there is $U ∈ \Oc$ with $A ≤ U$ and $∀i : U ≤ U_i$. As we saw above, this implies that $A ≤ e(U)$ and also $∀i : e(U) ≤ U_i$.
		✦[\ref{ax:IS-O}] This axiom is dual to \ref{ax:IS-K}.
		✦[\ref{ax:IS-inter}] Suppose that $A ≤ U$ with $A ∈ \fix(f)$ and $U ∈ \fix(e)$. Then we find $V∈\Oc$ and $B∈\Kc$ such that $A ≤ V ≤ B ≤ U$. Hence $A ≤ e(V)$ and $f(B) ≤ U$. Finally, we show that $e(V) ≤^* f(B)$: for all $C ∈ \fix(f)$ such that $C ≤ e(V)$, we have $C = f(C) ≤ V ≤ B$, hence $C = f(C) ≤ f(B)$.
	\end{itemize}
	We also showed that for all $A ∈ \fix(f) ⊆ \Kc$, we have $A ≤ U ⇔ A ≤ e(U)$. This shows that $(\Oc ↠ \fix(e), \fix(f) ↣ \Kc)$ is a morphism in $\cInter_{αβ}$, so the idempotent $(e,f)$ splits as announced.
\end{proof*}

From Proposition~\ref{prop:inter-cauchy-compl} above, we can present $(α,β)$-intermediate structures using ``proximity posets,'' to copy the terminology of, e.g., \cite{JunSun1996}, or ``infosyses'' in the terminology of \cite{Vic1993}. A \emph{proximity poset} is a poset $P$ equipped with an upward binary relation $[≺] ⊆ P^\op × P$ which is idempotent and, optionally, is contained in the identity $[≤] ⊆ P^\op × P$ if we want to work with idempotent comonads. An \emph{$(α,β)$-proximity poset} is a proximity poset $(P,≺)$ such that $p≺(—)$ is an $α$-filter for all $p ∈ P$, and such that $(—)≺ p$ is a $β$-ideal for all $p∈P$. A \emph{morphism} $P → Q$ of $(α,β)$-proximity posets is an up-set $R ⊆ Q^\op × P$ such that:
\begin{itemize}
	✦ $R(—,p)$ is a $β$-ideal for all $p∈P$.
	✦ $R(q,—)$ is an $α$-filter for all $q∈Q$.
	✦ $R = [≺] ∘ R ∘ [≺]$.
\end{itemize}
Then $\cInter_{αβ}$ is equivalent to the category of $(α,β)$-proximity posets.
The $β$-continuous poset $\Oc_β(P,≺)$ presented by $(P,≺)$ is the set of $β$-ideals $I ⊆ P$ such that for all $p ∈ P$, we have $p ∈ I ⇔ ∃q∈I : p ≺ q$. Equivalently, $I ∈ \Oc_β(P,≺)$ if it is a subset of $P$ satisfying:
\begin{itemize}
	✦ $p ∈ I$ and $q ≺ p$ imply that $q ∈ I$.
	✦ $(p_i)_i ⊆ I$ is $β$-small imply that there is $p ∈ I$ with $∀i : p_i ≺ p$.
\end{itemize}
The $(α,β)$-dual $\Kc_α(P,≺)^\op$ of $\Oc_β(P,≺)$ is the set of $α$-filters $F ⊆ P$ such that for all $p ∈ P$, we have $p ∈ F ⇔ ∃q∈F : q ≺ p$. There are two canonical maps from $P$ to the intermediate structure of $(\Oc_β(P,≺),\Kc_α(P,≺))$:
\begin{itemize}
	✦ $o : P → \Oc_β(P,≺)$ sends $p$ to $\setst{q}{q≺p}$.
	✦ $k : P → \Kc_α(P,≺)$ sends $p$ to $\setst{q}{p≺q}$.
\end{itemize}
Note that $p ≺ q ⇔ k(p) ≤ o(q)$, which is reminiscent of the relation between polarities and the complete lattices their present. We also have $o(p)≤k(p)$ for all $p∈P$. Given $I ∈ \Oc_β(P,≺)$ and $p ∈ P$, we have $p∈I ⇔ k(p) ≤ I$. Given $I ∈ \Oc_β(P,≺)$ and $F ∈ \Kc_α(P,≺)$, we have
\[ F≤I ⇔ ∃p ∈ P : F≤o(p)≤k(p)≤I \text{.} \]
Consequently, $I ≪_β J$ for $I,J ∈ \Oc_β(P,≺)$ if and only if there is $p∈P$ such that $I ≤ o(p) ≤ k(p) ≤ J$, i.e., such that $I ⊆ o(p)$ and $p ∈ J$.
The two maps $o$ and $k$ coincide when $[≺] = [≤]$ and $(\Oc_β(P,≺),\Kc_α(P,≺)) = (\Idl_β(P),\Flt_α^\op(P))$. The proximity poset $(P,≺)$ build in the proof of Proposition~\ref{prop:inter-cauchy-compl} to present $(\Oc,\Kc)$ is given by $P = \setst{(U,A) ∈ \Oc×\Kc}{U≤A}$ and $(U,A) ≺ (U',A') ⇔ A≤U'$. The map $o : P→\Oc$ sends $(U≤A)$ to $U$. Dually, $k : P→\Kc$ sends $(U≤A)$ to $A$.

\subsection{Proximity lattices for $\cSt_α\cInf_{α'}\cCont_{β'}\cSup_β$}

We specialize proximity posets to $\cSt_α\cInf_{α'}\cCont_{β'}\cSup_β$. The next lemma shows that $\cSt_α\cInf_{α'}\cCont_{β'}\cSup_β$ is Cauchy-complete.

\begin{lem}{}{cocompl-cauchy}
	$β$-cocomplete posets are stable under poset retracts.
\end{lem}

\begin{proof*}{}
	Let $P$ be a $β$-cocomplete poset and let $e : P→P$ be an idempotent poset map. The supremum of a $β$-small family $(x_i)_i ⊆ \fix(e)$ is computed as $e(\sjoin{β}_i x_i)$.
\end{proof*}

\begin{dfn}{}{}
	An $(α,β)$-proximity poset $(P,≺)$ is \emph{exact} when for all $β$-small $(x_i)_i ⊆ P$ and for all $α$-small $(y_j)_j ⊆ P$,
	\[ ∀i,j : x_i ≺ y_j ⇒ ∃ z ∈ P : (∀i : x_i ≺ z) ∧ (∀j : z ≺ y_j) \text{.} \]
\end{dfn}

\begin{lem}{}{exact-prox}
	An $(α,β)$-proximity poset $(P,≺)$ is exact iff $(\Oc_β(P,≺),\Kc_α(P,≺))$ is exact.
\end{lem}

\begin{proof*}{}
	Suppose that $(\Oc_β(P,≺),\Kc_α(P,≺))$ is exact. We show that $P$ is exact.\footnote{A direct proof using the existence of $α$-meets in $\Oc_β(P,≺)$ would also be possible.} Let $(x_i)_i ⊆ P$ be a $β$-small family and let $(y_j)_j ⊆ P$ be an $α$-small family such that $x_i ≺ y_j$ for all $i,j$. Thus $\sjoin{β}_i k(x_i) ≤ \smeet{α}_j o(y_j)$ and we obtain $z ∈ P$ such that $\sjoin{β}_i k(x_i) ≤ o(z)$ and $k(z) ≤ \smeet{α}_j o(y_j)$, i.e., such that $∀i : x_i≺z$ and $∀j : z≺y_j$.
	
	Reciprocally, suppose that $(P,≺)$ is $(α,β)$-exact. Let $(I_j)_j ⊆ \Oc_β(P,≺)$ be an $α$-small family. We show that $⋂_j I_j ∈ \Oc_β(P,≺)$: if $(x_i)_i ⊆ ⋂_j I_j$ is a $β$-small family, then for each $j$ there is $y_j ∈ I_j$ with $∀i : x_i≺y_j$. Since $P$ is $(α,β)$-exact, there is $z ∈ P$ with $∀i : x_i≺z$ and $∀j : z≺y_j$. Thus $z$ is a $≺$-upper bound of $\set{x_i}_i$ in $⋂_j I_j$. Moreover, if $x ∈ ⋂_j I_j$ and $y ≺ x$, then $y ∈ ⋂_j I_j$. This shows that $⋂_j I_j ∈ \Oc_β(P,≺)$.
\end{proof*}

\begin{cor}{}{stinfcontsup-proximity-pres}
	$\cSt_α\cInf_{α'}\cCont_{β'}\cSup_β$ is equivalent to the category of $α$-complete and $β$-cocomplete exact $(α',β')$-proximity posets.
\end{cor}

\begin{proof*}{}
	Let $(P,≺)$ be an $α$-complete, $β$-cocomplete, exact $(α',β')$-proximity poset. Then $\Idl_{β'}(P)$ has $β$-joins: the $β$-join of $(I_k)_k$ is the down-closure of the closure of $⋃_k I_k$ under $β$-joins. Dually, $\Flt_{α'}^\op(P)$ has $α$-meets. By Lemmas~\ref{lem:cocompl-cauchy} and \ref{lem:exact-prox}, $\Oc_β(P,≺)$ is in $\cSt_α\cInf_{α'}\cCont_{β'}\cSup_β$.
	
	Let $\Oc ∈ \cSt_α\cInf_{α'}\cCont_{β'}\cSup_β$ and let $\Kc^\op$ be its $(α',β')$-dual. Let $P = \setst{(U,A) ∈ \Oc×\Kc}{U≤A}$, equipped with the $(α',β')$-proximity $(U,A) ≺ (U',A') ⇔ A≤U'$ from the proof of Proposition~\ref{prop:inter-cauchy-compl} so that $\Oc = \Oc_{β'}(P,≺)$. From Lemma~\ref{lem:exact-prox}, we know that $(P,≺)$ is exact. It remains to show that $P$ has $α$-meets and, dually, $β$-joins. Indeed, $\Oc$ has $α$-meets and $\Kc$ too. Moreover, if $(U_i ≤ A_i)_i ⊆ P$ is $α$-small, then $\smeet{α}_i U_i ≤ \smeet{α}_i A_i$ by Lemma~\ref{lem:polarity-sep-inf}.
\end{proof*}

\begin{rmk}{}{}
	When $α=α'$, every $α$-complete $(α',β')$-proximity poset is exact. In this case, $\cSt_α\cInf_{α'}\cCont_{β'}\cSup_β$ is the Cauchy-completion of its subcategory of algebraic posets (this also works, dually, when $β=β'$).
\end{rmk}

\paragraph{Strong proximity lattices}
The notion of \emph{strong} proximity lattice from \cite[Sec.~5]{JunSun1996} also carries over. Let $(P,≺)$ be an $α$-complete, $β$-cocomplete, exact $(α',β')$-proximity poset. The map $o : P → \Oc_{β'}(P,≺)$ always preserves $α$-meets, using that $x ≺ \smeet{α}_i y_i ⇔ ∀i : x ≺ y_i$. We say that $(P,≺)$ \emph{respects $β$-joins} if the map $o : P → \Oc_{β'}(P,≺)$ preserves also $β$-joins. Dually, it \emph{respects $α$-meets} if the map $k : P → \Kc_{α'}(P,≺)$ preserves $α$-meets. For instance, the canonical proximity poset associated to $\Oc ∈ \cSt_α\cInf_{α'}\cCont_{β'}\cSup_β$ in Corollary~\ref{cor:stinfcontsup-proximity-pres} respects $β$-joins and $α$-meets. Just like in \cite[Prop.~17]{JunSun1996}, $(P,≺)$ respects $β$-suprema if and only if for all $p ∈ P$ and all $β$-small family $(q_i)_i ⊆ P$,
\[ p ≺ \sjoin{β}_i q_i ⇒ ∃ (p_i ≺ q_i)_i ⊆ P : p ≺ \sjoin{β}_i p_i \text{.} \]
This comes directly from the way that $β$-joins are computed in $\Oc_β(P,≺)$ using Lemma~\ref{lem:cocompl-cauchy}. Like in \cite[Sec.~7]{JunSun1996}, we can characterize the morphisms in $\cSt_α\cInf_{α'}\cCont_{β'}\cSup_β$ that preserve $β$-joins in terms of proximity posets that respect $β$-joins. We record that here without proof, as it is exactly analogous to \cite{JunSun1996}:

\begin{prop}{}{}
	Let $R ⊆ Q^\op × P$ be a morphism between two $β$-cocomplete exact $(α',β')$-proximity poset. Suppose that $P$ respects $β$-suprema. Then the map $\Oc_{β'}(P) → \Oc_{β'}(Q)$ corresponding to $R$ preserves $β$-suprema if and only if for all $q ∈ Q$ and all $β$-small $(p_i)_i ⊆ P$,
	\[ q\ R\ \sjoin{β}_i p_i ⇒ ∃(q_i)_i : q ≺ \sjoin{β}_i q_i \text{ and } ∀i : q_i\ R\ p_i\text{.} \]
\end{prop}

\subsection{Perfect maps}

There are several important kinds of morphisms between compact ordered spaces: closed ordered relations, semicontinuous maps (upper and lower), and continuous maps.
In terms of the frames of open up-sets, closed relations correspond to preframe morphisms (\cite[Ch.~4]{Tow96}, \cite{JunKegMos2001}; see also \cite{KurMosJun2023}, \cite[Sec.~1.3.3]{Mar2023}), semicontinuous maps correspond to frame morphisms,
and continuous maps correspond to \emph{perfect} frame morphisms.

\begin{dfn}{}{}
	We say that a map of $β$-continuous posets is \emph{perfect} if it preserves the $≪_β$ relation.
\end{dfn}

\begin{prop}{}{perfect-iff-ladj}
	Let $f = (f^* ⊢ f_!) : (\Oc_1,\Kc_1) → (\Oc_2,\Kc_2)$ be a morphism in $\cInter_{αβ}$. The following are equivalent:
	\begin{itemize}
		✦ $f^* : \Oc_1 → \Oc_2$ is perfect.
		✦ $f_! : \Kc_2 → \Kc_1$ has a right adjoint $\Kc_1 → \Kc_2$ in the category of posets.
		✦ $f^*$ extends to a poset map $\Oc_1 ∪ \Kc_1 → \Oc_2 ∪ \Kc_2$ sending $\Kc_1$ to $\Kc_2$.
	\end{itemize}
\end{prop}

\begin{proof*}{}
	Suppose that $f^*$ is perfect. If $I ⊆ \Oc_1$ is a $β$-open $α$-filter, then ${↑}f^*[I]$ is also a $β$-open $α$-filter: it is an $α$-filter because $f$ is order-preserving, and it is $β$-open using Lemma~\ref{lem:charact-open}. This defines a map $g : \Kc_1 → \Kc_2$ and it is right adjoint to $f_!$:
	\[ f_!(B) ≤ A ⇔ ∀U≥A : B≤f^*(U) ⇔ B ≤ g(A) \text{.} \]
	
	Suppose that $f_! : \Kc_2 → \Kc_1$ has a right adjoint $g : \Kc_1 → \Kc_2$. We show that $g$ extends $f^* : \Oc_1 → \Oc_2$ to a poset map $\Oc_1 ∪ \Kc_1 → \Oc_2 ∪ \Kc_2$. Suppose $A ≤ U$. Then $g(A) ≤ f^*(U)$ because $f_!(g(A)) ≤ A ≤ U$. Suppose $U ≤ A$. By definition of the polarity order, $f^*(U) ≤ g(A) ⇔ ∀B≤f^*(U) : B≤g(A)$. This is equivalent to $∀B : f_!(B)≤U ⇒ f_!(B)≤A$ which is indeed true because $U≤A$.
	
	Suppose that $f^*$ extends to a poset map $g : \Oc_1 ∪ \Kc_1 → \Oc_2 ∪ \Kc_2$ sending $\Kc_1$ to $\Kc_2$. We show that $f^*$ is perfect. If $U ≪_β V$, then $U ≤ A ≤ V$ for some $A$, so that $f^*(U) ≤ g(A) ≤ f^*(V)$ and $f^*(U) ≪_β f^*(V)$.
\end{proof*}

\begin{cor}{}{}
	A morphism in $\cSt_α\cSup_β$ has a right adjoint (in the category $\cSt_α\cSup_β$) if and only if it is perfect and preserves all suprema.
\end{cor}

\begin{proof*}{}
	Let $f^* : \Oc_1 → \Oc_2$ be a morphism in $\cSt_α\cSup_β$ and let $f_! : \Kc_2 → \Kc_1$ be (the opposite of) its $(α,β)$-dual. If $f^*$ has a right adjoint, then it preserves all suprema and $f_!$ also has a right adjoint. By Proposition~\ref{prop:perfect-iff-ladj}, $f^*$ is perfect.
	
	Reciprocally, suppose that $f^*$ is perfect and preserves all suprema. Let $g : \Kc_1 → \Kc_2$ be the right adjoint of $f_!$ in the category of posets. We show that $g$ preserves all $β$-suprema, so that it is (the opposite of) a morphism in $\cSt_β\cSup_α$ adjoint to $f_!$. We show that $g\p[\big]{\sjoin{β}_i A_i} ≤ \sjoin{β}_i g(A_i)$. Suppose that $\sjoin{β}_i g(A_i) ≤ U$. Then for all $i$, we have $g(A_i) ≤ U$ and thus, by the computation of $g$ in Proposition~\ref{prop:perfect-iff-ladj}, there is $V_i$ such that $A_i≤V_i$ and $f^*(V_i) ≤ U$. We obtain $\sjoin{β}_i A_i ≤ \sjoin{β}_i V_i$ and $\sjoin{β}_i f^*(V_i) = f^*\p[\big]{\sjoin{β}_i V_i} ≤ U$, hence $g\p[\big]{\sjoin{β}_i A_i} ≤ U$.
\end{proof*}

\newpage
\begin{exmps}{}{}~
	\begin{itemize}
		✦ The category of sets and relations is a full subcategory of $\cSt_0\cSup_0$, via $X ↦ 𝒫(X)$. The adjoints are the graphs of functions. This generalizes to posets and ordered relations, which is the category of algebraic posets of $\cSt_0\cSup_0$.
		✦ The category of compact ordered spaces and closed ordered relations is a full subcategory of $\cSt_ω\cSup_ω$ (\cite[Ch.~4]{Tow96}, \cite{JunKegMos2001}). The adjunctions in this category correspond to the continuous order-preserving maps.
		✦ The frame of opens of a poset $P$ with its Alexandroff topology is $\Up(P) = \Flt_0(P)$. The soberification of $P$, dual to $\Flt_0(P)$, is $\Idl_ω(P)$. The pairs $(\Idl_ω(P),\Flt_0(P))$ are the algebraic $(0,ω)$-intermediate structures. The perfect maps $\Idl_ω(P) → \Idl_ω(Q)$ in $\cFlat_0\cCont_ω$ correspond to the poset maps $P→Q$. They are also called ``essential geometric morphisms'' in topos theory. Note that $\cFlat_0\cCont_ω = \cSt_0\cCont_ω$ is the category of continuous posets with Scott-continuous maps, and that $\cFlat_ω\cCont_0 = \cInf_ω\cSup_0$ is the category of completely distributive frames.
		✦ Let $P$ and $Q$ be $α$-complete $β$-cocomplete posets. The perfect maps $\Idl_β(P) → \Idl_β(Q)$ in $\cSt_α\cSup_β$ correspond to the poset maps $P→Q$ preserving $α$-meets. The perfect maps $\Idl_β(P) → \Idl_β(Q)$ that furthermore preserve $β$-suprema correspond to the poset maps $P→Q$ that preserve $α$-meets and $β$-joins. Thus the adjunctions in $\cSt_α\cSup_β$ generalize the morphisms of $α$-complete $β$-cocomplete posets (which is the ``algebraic'' case).
	\end{itemize}
\end{exmps}

\includebibliography


@article{Kaw2021,
	title = {Predicative theories of continuous lattices},
	author = {Kawai, Tatsuji},
	year = 2021,
	month = may,
	journal = {Logical Methods in Computer Science},
	volume = {Volume 17, Issue 2},
	publisher = {Episciences.org}
}

@article{Law1979,
	title = {The duality of continuous posets},
	author = {Lawson, Jimmie D.},
	year = 1979,
	month = jan,
	journal = {Houston Journal of Mathematics},
	volume = {5}
}

@unpublished{FusGeh2025,
	title = {Canonical Extensions Quickly},
	author = {Fussner, Wesley and Gehrke, Mai},
	year = 2025,
	month = aug,
	url = {https://hal.science/hal-05231128v1}
}

@article{Tow2006,
	title = {On the parallel between the suplattice and preframe approaches to locale theory},
	author = {Townsend, C. F.},
	year = 2006,
	month = jan,
	journal = {Annals of Pure and Applied Logic},
	volume = {137},
	number = {1},
	pages = {391--412}
}

@book{MoeVer2000,
	title = {Proper Maps of Toposes},
	author = {Moerdijk, Ieke and Vermeulen, Jacob Johan Caspar},
	year = 2000,
	series = {Memoirs of the American Mathematical Society},
	volume = {148},
	publisher = {American Mathematical Society}
}

@article{Ver1994,
	title = {Proper maps of locales},
	author = {Vermeulen, J. J. C.},
	year = 1994,
	month = feb,
	journal = {Journal of Pure and Applied Algebra},
	volume = {92},
	number = {1},
	pages = {79--107}
}

@article{Nov1982,
	title = {Generalization of Continuous Posets},
	author = {Novak, Dan},
	year = 1982,
	journal = {Transactions of the American Mathematical Society},
	volume = {272},
	number = {2},
	eprint = {1998719},
	pages = {645--667},
	publisher = {American Mathematical Society}
}

@article{Bar1996,
	title = {$\Zc$-continuous posets},
	author = {Baranga, Andrei},
	year = 1996,
	month = may,
	journal = {Discrete Mathematics},
	volume = {152},
	number = {1},
	pages = {33--45}
}

@article{Vic1993,
	title = {Information systems for continuous posets},
	author = {Vickers, Steven},
	year = 1993,
	month = jun,
	journal = {Theoretical Computer Science},
	volume = {114},
	number = {2},
	pages = {201--229}
}

@article{RosWoo1994,
	title = {Constructive complete distributivity IV},
	author = {Rosebrugh, Robert and Wood, R. J.},
	year = 1994,
	month = jun,
	journal = {Applied Categorical Structures},
	volume = {2},
	number = {2},
	pages = {119--144}
}

@book{GieHofKeiLawMisSco2003,
	title = {Continuous Lattices and Domains},
	author = {Gierz, G. and Hofmann, K. H. and Keimel, K. and Lawson, J. D. and Mislove, M. and Scott, D. S.},
	year = 2003,
	series = {Encyclopedia of Mathematics and its Applications},
	publisher = {Cambridge University Press},
	address = {Cambridge}
}

@book{GehvGoo2024,
	title = {Topological Duality for Distributive Lattices: Theory and Applications},
	shorttitle = {Topological Duality for Distributive Lattices},
	author = {Gehrke, Mai and {van Gool}, Sam},
	year = 2024,
	month = mar,
	publisher = {Cambridge University Press},
	shorthand = {GvG24}
}

@misc{AbbJun2026,
	title = {On the symmetry behind duality},
	author = {Abbadini, Marco and Jung, Achim},
	year = 2026,
	month = jul,
	number = {arXiv:2507.18245},
	eprint = {2507.18245},
	primaryclass = {math.LO},
	publisher = {arXiv},
	archiveprefix = {arXiv}
}

@inproceedings{Hof1981,
	title = {Continuous posets, prime spectra of completely distributive complete lattices, and Hausdorff compactifications},
	booktitle = {Continuous Lattices},
	author = {Hoffmann, Rudolf-E.},
	editor = {Banaschewski, Bernhard and Hoffmann, Rudolf-Eberhard},
	year = 1981,
	pages = {159--208},
	publisher = {Springer},
	address = {Berlin, Heidelberg}
}

@inproceedings{Hof1981a,
	title = {Projective sober spaces},
	booktitle = {Continuous Lattices},
	author = {Hoffmann, Rudolf-E.},
	editor = {Banaschewski, Bernhard and Hoffmann, Rudolf-Eberhard},
	year = 1981,
	pages = {125--158},
	publisher = {Springer},
	address = {Berlin, Heidelberg}
}

@phdthesis{Jak17,
	title = {d-Frames as algebraic duals of bitopological spaces},
	author = {Jakl, Tom{\'a}{\v s}},
	year = 2017,
	school = {University of Prague and University of Birmingham}
}

@phdthesis{Kli2012,
	title = {A bitopological point-free approach to compactifications},
	author = {Klinke, Olaf Karl},
	year = 2012,
	school = {University of Birmingham}
}

@phdthesis{Tow96,
	title = {Preframe Techniques in Constructive Locale Theory},
	author = {Townsend, Christopher Francis},
	year = 1996,
	school = {University of London}
}

@article{BanErn1983,
	title = {The category of $\Zc$-continuous posets},
	author = {Bandelt, Hans-J. and Ern{\'e}, Marcel},
	year = 1983,
	month = dec,
	journal = {Journal of Pure and Applied Algebra},
	volume = {30},
	number = {3},
	pages = {219--226}
}

@article{JunSun1996,
	title = {On the Duality of Compact vs. Open},
	author = {Jung, Achim and S{\"u}nderhauf, Philipp},
	year = 1996,
	month = dec,
	journal = {Annals of the New York Academy of Sciences},
	volume = {806},
	number = {1},
	pages = {214--230}
}

@article{vanGool2012,
	title = {Duality and canonical extensions for stably compact spaces},
	author = {{van Gool}, Sam J.},
	year = 2012,
	month = jan,
	journal = {Topology and its Applications},
	volume = {159},
	number = {1},
	pages = {341--359},
	shorthand = {vG12}
}

@article{Kop1995,
	title = {Asymmetry and duality in topology},
	author = {Kopperman, Ralph},
	year = 1995,
	month = sep,
	journal = {Topology and its Applications},
	volume = {66},
	number = {1},
	pages = {1--39}
}

@article{Smy1992a,
	title = {Stable Compactification I},
	author = {Smyth, M. B.},
	year = 1992,
	journal = {Journal of the London Mathematical Society},
	volume = {s2-45},
	number = {2},
	pages = {321--340}
}

@inproceedings{GehVos2011,
	title = {A View of Canonical Extension},
	booktitle = {Logic, Language, and Computation},
	author = {Gehrke, Mai and Vosmaer, Jacob},
	editor = {Bezhanishvili, Nick and L{\"o}bner, Sebastian and Schwabe, Kerstin and Spada, Luca},
	year = 2011,
	pages = {77--100},
	publisher = {Springer},
	address = {Berlin, Heidelberg}
}

@article{Kel1963,
	title = {Bitopological Spaces},
	author = {Kelly, J. C.},
	year = 1963,
	journal = {Proceedings of the London Mathematical Society},
	volume = {s3-13},
	number = {1},
	pages = {71--89}
}

@article{Sua2022,
	title = {The category of finitary biframes as the category of pointfree bispaces},
	author = {Suarez, Anna Laura},
	year = 2022,
	month = feb,
	journal = {Journal of Pure and Applied Algebra},
	volume = {226},
	number = {2},
	pages = {106783}
}

@TechReport{Jung06,
	author = 	 {A. Jung and M. A. Moshier},
	title = 	 {On the bitopological nature of {S}tone duality},
	institution =  {University of Birmingham},
	year = 	 2006,
	number =	 {CSR-06-13}
}

@article{Ern1999,
	title = {$\Zc$-Continuous Posets and Their Topological Manifestation},
	author = {Ern{\'e}, Marcel},
	year = 1999,
	month = jun,
	journal = {Applied Categorical Structures},
	volume = {7},
	number = {1},
	pages = {31--70}
}

@inproceedings{Mar1981,
	title = {A motivation and generalization of Scott's notion of a continuous lattice},
	booktitle = {Continuous Lattices},
	author = {Markowsky, George},
	editor = {Banaschewski, Bernhard and Hoffmann, Rudolf-Eberhard},
	year = 1981,
	pages = {298--307},
	publisher = {Springer},
	address = {Berlin, Heidelberg}
}

@article{ErnZha2001,
	title = {$\Zc$-Join Spectra of $\Zc$-Supercompactly Generated Lattices},
	author = {Ern{\'e}, Marcel and Zhao, Dongsheng},
	year = 2001,
	month = jan,
	journal = {Applied Categorical Structures},
	volume = {9},
	number = {1},
	pages = {41--63}
}

@article{WriWagTha1978,
	title = {A uniform approach to inductive posets and inductive closure},
	author = {Wright, J. B. and Wagner, E. G. and Thatcher, J. W.},
	year = 1978,
	month = jan,
	journal = {Theoretical Computer Science},
	volume = {7},
	number = {1},
	pages = {57--77}
}

@article{YuaLi2019,
	title = {The Duality Theory of General $\Zc$-continuous Posets},
	author = {Yuan, Zhenzhu and Li, Qingguo},
	year = 2019,
	month = aug,
	journal = {Electronic Notes in Theoretical Computer Science},
	series = {The proceedings of ISDT 2019, the 8th International Symposium on Domain Theory and Its Applications, ISDT 2019},
	volume = {345},
	pages = {281--292}
}

@incollection{HofMis1985,
	title = {Free Objects in the Category of Completely Distributive Lattices},
	booktitle = {Continuous Lattices and Their Applications},
	author = {Hofmann, Karl H. and Mislove, Michael},
	year = 1985,
	publisher = {CRC Press}
}

@inproceedings{Law1987,
	title = {The versatile continuous order},
	booktitle = {Mathematical Foundations of Programming Language Semantics},
	author = {Lawson, Jimmie D.},
	editor = {Main, M. and Melton, A. and Mislove, M. and Schmidt, D.},
	year = 1988,
	pages = {134--160},
	publisher = {Springer},
	address = {Berlin, Heidelberg}
}

@article{JunKegMos2001,
	title = {Stably Compact Spaces and Closed Relations},
	author = {Jung, Achim and Kegelmann, Mathias and Moshier, M. Andrew},
	year = 2001,
	month = nov,
	journal = {Electronic Notes in Theoretical Computer Science},
	series = {MFPS 2001,Seventeenth Conference on the Mathematical Foundations of Programming Semantics},
	volume = {45},
	pages = {209--231}
}

@incollection{KurMosJun2023,
	title = {Stone Duality for~Relations},
	booktitle = {Samson Abramsky on Logic and Structure in Computer Science and Beyond},
	author = {Kurz, Alexander and Moshier, Andrew and Jung, Achim},
	editor = {Palmigiano, Alessandra and Sadrzadeh, Mehrnoosh},
	year = 2023,
	pages = {159--215},
	publisher = {Springer International Publishing},
	address = {Cham}
}

@phdthesis{Mar2023,
	title = {Categorical logic from the perspective of duality and compact ordered spaces},
	author = {Marqu{\`e}s, J{\'e}r{\'e}mie},
	year = 2023,
	month = sep,
	collaborator = {Gehrke, Mai},
	school = {Universit\'e C\^ote d'Azur}
}

@article{Ern2016,
	title = {Choice-Free Dualities for Domains},
	author = {Ern{\'e}, Marcel},
	year = 2016,
	month = oct,
	journal = {Applied Categorical Structures},
	volume = {24},
	number = {5},
	pages = {471--496}
}

@book{BalDwi1974,
	title = {Distributive Lattices},
	author = {Balbes, Raymond and Dwinger, Philip},
	year = 1974,
	publisher = {University of Missouri Press},
	address = {Columbia, Miss.}
}

@article{GhiMel1997,
	title = {Constructive canonicity in non-classical logics},
	author = {Ghilardi, Silvio and Meloni, Giancarlo},
	year = 1997,
	month = jun,
	journal = {Annals of Pure and Applied Logic},
	volume = {86},
	number = {1},
	pages = {1--32}
}

@article{DunGehPal2005,
	title = {Canonical Extensions and Relational Completeness of Some Substructural Logics},
	author = {Dunn, J. Michael and Gehrke, Mai and Palmigiano, Alessandra},
	year = 2005,
	journal = {The Journal of Symbolic Logic},
	volume = {70},
	number = {3},
	eprint = {27588391},
	pages = {713--740},
	publisher = {Association for Symbolic Logic}
}

@article{GehHar2001,
	title = {Bounded Lattice Expansions},
	author = {Gehrke, Mai and Harding, John},
	year = 2001,
	month = apr,
	journal = {Journal of Algebra},
	volume = {238},
	number = {1},
	pages = {345--371}
}

@article{GehJanPal2013,
	title = {$\Delta_1$-completions of a Poset},
	author = {Gehrke, Mai and Jansana, Ramon and Palmigiano, Alessandra},
	year = 2013,
	month = mar,
	journal = {Order},
	volume = {30},
	number = {1},
	pages = {39--64}
}
\end{document}